\documentclass[preprint,11pt]{elsarticle}

\usepackage{amsmath,enumerate}

\usepackage{amssymb,amsopn}
\usepackage{graphicx}
\usepackage{caption}
\usepackage{subcaption}
\usepackage{mathtools}
\usepackage{array}
\usepackage{arydshln}

\usepackage{amsmath}
\usepackage{amsthm}
\theoremstyle{definition}

\usepackage{pdflscape}
\usepackage{afterpage}
\usepackage{longtable}
\usepackage{appendix}

\newcommand{\mcalK}{\mathcal{K}}

\newcommand{\Lap}{L}
\usepackage{tikz}

\usepackage{geometry}
\usepackage{tikz}
\usetikzlibrary{calc,shapes}
\usepackage{calc}
\usepackage{amsmath}

\newtheorem{theorem}{Theorem}
\newtheorem{lemma}[theorem]{Lemma}
\newtheorem{proposition}[theorem]{Proposition}
\newtheorem{corollary}[theorem]{Corollary}

\newtheorem{remark}[theorem]{Remark}
\setcitestyle{square,comma}

\begin{document}
\title{\bf Optical tomography on graphs}
\author[ukentucky]{Francis J. Chung}
\ead{fj.chung@uky.edu}
\author[umich_math]{Anna C. Gilbert}
\ead{annacg@umich.edu}
\author[umich_math]{Jeremy G. Hoskins \corref{cor1}}
\ead{jhoskin@umich.edu}
\author[umich_math,umich_phys]{John C. Schotland}
\ead{schotland@umich.edu}

\address[ukentucky]{Department of Mathematics, University of Kentucky, Lexington, KY 40506\\}
\address[umich_math]{Department of Mathematics, University of Michigan, Ann Arbor, MI 48109\\}
\address[umich_phys]{Department of Physics, University of Michigan, Ann Arbor, MI 48109\\}

\cortext[cor1]{Corresponding author}

\begin{abstract}

We present an algorithm for solving inverse problems on graphs analogous to those arising in diffuse optical tomography for continuous media. In particular, we formulate and analyze a discrete version of the inverse Born series, proving estimates characterizing the domain of convergence, approximation errors, and stability of our approach. We also present a modification which allows additional information on the structure of the potential to be incorporated, facilitating recovery for a broader class of problems.

\end{abstract}

\maketitle

\section{Introduction}
Inverse problems arise in numerous settings within discrete mathematics, including graph tomography~\cite{vardi_tom,Patch1999,tom_fin_graph,grun,CHUNG2001} and resistor networks~\cite{Curtis1994,morrow_90,morrow_91,morrow,ingerman,liouville,liliana_resistor}. In such problems, one is typically interested in reconstructing a function defined on edges of a fixed graph or, in some cases, the edges themselves. In this paper, we focus on recovering vertex properties of a graph from boundary measurements. The problem we consider is the discrete analog of optical tomography. Optical tomography is a
biomedical imaging modality that uses scattered light as a probe of structural variations in the optical properties of tissue~\cite{medic_imag}. The inverse problem of optical tomography consists of recovering the coefficients of a Schrodinger operator from boundary measurements.

Let $G=(V,E)$ be a finite locally connected loop-free graph with vertex boundary $\delta V$.  We consider the time-independent diffusion equation \cite{ip2008} 
\begin{eqnarray}
\label{diff_eq}
&({\Lap} u ) (x)  + \alpha_0 [1+\eta(x)] u(x) = f(x),  \quad x \in V , \\
&t \, u(x) + \partial u(x) = g(x),  \quad x \in \delta V ,
\label{bc}
\end{eqnarray}
which, in the continuous setting, describes the transport of the energy density of an optical field in an absorbing medium. Here we assume that the absorption of the medium is nearly constant, with small absorbing inhomogeneities represented by an {absorption coefficient}, or {\it vertex potential} $\eta.$ In place of the Laplace-Beltrami operator, we introduce the combinatorial Laplacian $\Lap$ defined by
\begin{equation}
(\Lap u)(x) = \sum_{y \sim x} \left[ u(x) - u(y) \right], 
\end{equation}
where $y \sim x$ if the vertices $x$ and $y$ are adjacent. We make use of the graph analog of Robin boundary conditions, where 
\begin{equation}
\partial u (x) = \sum_{\substack{y \in V \\ y \sim x}} \left[ u(x) - u(y)\right],
\end{equation}
and $t$ is an arbitrary nonnegative parameter, which interpolates between Dirichlet and Neumann boundary conditions. If the vertex potential $\eta$ is non-negative, then there exists a unique solution to the diffusion equation~(\ref{diff_eq}) satisfying the boundary
condition~(\ref{bc}),~[see \citenum{me_arxiv} and the references therein]. 

In \cite{me_arxiv} we presented an algorithm for solving the forward problem of determining $u,$ given $\eta.$ Our approach was a perturbative one, making use of known Green's functions for the time-independent diffusion equation (or Schr\"{o}dinger equation) \cite{part_dat,journ_func,vert,eig_path,disc_op,polyom,pert,cartier,soardi,inf_yamasaki}, with $\eta$ identically zero. The corresponding inverse problem, which we refer to as graph optical tomography, is to recover the potential $\eta$ from measurements of $u$ on the boundary of the graph. More precisely, let $G = (V,E)$ be a connected subgraph of a finite graph $\Gamma = (\mathcal{V}, \mathcal{E})$ and let $\delta V$ denote those vertices in $\mathcal{V}$ adjacent to a vertex in $V.$ In addition, let $S,R$ denote fixed subsets of $\delta V$. We will refer to elements of $S$ and $R$ as sources and receivers, respectively. For a fixed potential $\eta,$  source $s\in S$ and receiver $r \in R,$ let $u(r,s;\eta)$ be the solution to (\ref{diff_eq}) with vertex potential $\eta$ and boundary condition (\ref{bc}), where
\begin{equation}
g(x) = 
\begin{cases} 1 & x = s,\\
0 & x \neq s.
\end{cases}
\end{equation}
We define the Robin-to-Dirichlet map $\Lambda_\eta$ by
\begin{equation}
\Lambda_\eta(s,r) = u(r,s;\eta).
\end{equation}
The inverse problem is to recover $\eta$ from the Robin-to-Dirichlet map $\Lambda_\eta.$

Eqs.~(\ref{diff_eq}) and (\ref{bc}) also arise when considering the Schr\"{o}dinger equation on graphs and related inverse problems~\cite{oberlin,kyo,ando,liouville}. For circular planar graphs, or lattice graphs in two or more dimensions, these works outline an algorithm that can be used to recover the vertex potential. In particular, the first three employ special combinations of boundary sources which force the solution in the interior to be zero except on a small, controllable set of vertices. Using this approach, the potential at each vertex can be calculated. Then, starting at the boundary, the entire potential can be recovered. The resulting algorithm relies on the lattice structure of the graphs and is unstable for potentials with large support.

In this paper we present a reconstruction method for graph optical tomography that is based on inversion of the Born series solution to the forward problem~\cite{ip2008, vasquez, Markel:03,inv_brn_num,inv_brn_opt, inv_brn_calderon, inv_brn_scalar}. Using this approach, we show that it is possible to recover vertex potentials for a general class of graphs under certain smallness conditions on the boundary measurements. In addition, we obtain sufficient conditions under which the inverse Born series converges to the vertex potential. We also obtain a corresponding stability estimate, which is independent of the support of the potential. In numerical studies of the inverse Born series for large potentials or large graphs, where exact recovery is not guaranteed, we nevertheless find that good qualitative recovery of large scale features of the potential is possible. Moreover, our approach can be easily modified to incorporate additional information on the structure of the potential, improving both the speed and accuracy of the algorithm. As an application of this idea, we show how to determine the potential $\eta$ using data for multiple values of $\alpha_0$, assuming $\eta$ is independent of $\alpha_0.$ This allows us to apply our method to graphs whose structure makes exact potential recovery otherwise impossible.

The remainder of this paper is organized as follows. In Section \ref{sec:IB_o} we briefly review key results on the solvability of the forward problem and introduce the Born series. We obtain necessary conditions for the convergence of the inverse Born series depending on the measurement data and the graph. We also describe related stability and error estimates.
In Section \ref{sec:imp} we discuss the numerical implementation of the inverse series and present the results of numerical simulations. Finally, in Section \ref{sec:mult_freq} we extend our results to the case where measurements can be taken at multiple values of $\alpha_0$.

\section{Inverse Born series}\label{sec:IB_o}
\subsection{Forward Born series}
In this section we formulate the inverse Born series. We begin by reviewing some important properties of the Born series, based in part on \cite{me_arxiv, ip2008}.

We recall that the {\it background Green's function} \cite{me_arxiv} for (\ref{diff_eq}) is the matrix $G_0$ whose $i,j$th entry is the solution to (\ref{diff_eq}), with $\eta \equiv 0$, at the $i$th vertex for a unit source at the $j$th vertex. Under suitable restrictions this matrix can be used to construct the {\it Robin-to-Dirichlet map} $\Lambda_{\eta}$ giving the solution of (\ref{diff_eq}) on $R \subset \delta V$ to unit sources located in $S \subset \delta V.$ To write a compact expression for $\Lambda_\eta$ in terms of $G_0,$ let $D_{\eta}$ denote the matrix with entries given by
\begin{equation}\nonumber
\left(D_{\eta}\right)_{i,j} = \begin{cases}
\eta_i &\text{if } i = j,\\
0 & \text{else}.
\end{cases}
\end{equation}
Additionally, for any two sets $U,W \subset V \cup \delta V,$ let $G_0^{U;W}$ denote the submatrix of $G_0$ formed by taking the rows indexed by $U$ and the columns indexed by $W.$ For $\eta$ sufficiently small we may write the Robin-to-Dirichlet map as a Neumann series
\begin{equation}\label{eq:f_born_r2dmap}
\Lambda_\eta(s,r) = G_0(r,s) - \sum_{j=1}^\infty K_j (\eta,\cdots,\eta) \,(r,s),\quad  r \in R, \,s \in S,
\end{equation}
where $K_j : \ell^p(V^n) \rightarrow \ell^p(R \times S)$ is defined by
\begin{equation}
K_j(\eta_1,\cdots,\eta_j) \,(r,s)=(-\alpha_0)^j G_0^{r;V} D_{\eta_1} \,G_0^{V;V} \,D_{{\eta}_2} \cdots G_0^{V;V} D_{{\eta}_j} G_0^{V;s}.
\end{equation}
We refer to the series (\ref{eq:f_born_r2dmap}) as the {\it forward Born series}.

In order to establish the convergence and stability of (\ref{eq:f_born_r2dmap}), we seek appropriate bounds on the operators $K_j: \ell^p(V\times \cdots \times V) \rightarrow \ell^p (\delta V \times \delta V).$ Note that if $|V|$ and $|\delta V|$ are finite then all norms are equivalent. 

However, since we are interested in the rate of convergence of the inverse series it will prove useful to establish bounds for arbitrary $\ell_p$ norms.

\begin{proposition}\label{thm_fwd_brn_thm}
Let $p,q  \in [1,\infty]$ such that $1/p + 1/q =1$ and define the constants $\nu_p$ and $\mu_p$ by
\begin{equation}\label{eq:forward_constants_def}
\nu_p =\alpha_0 \|G_0^{R;V}\|_{\ell^q(V) \times \ell^p(R)} \|G_0^{V;S}\|_{\ell^q(V) \times \ell^p(S)},\quad {\rm and}\quad \mu_p = \alpha_0\, C_{G_0^{V;V},q},
\end{equation}
where
\begin{equation}
C_{G_0^{V;V},q} = {\rm max}_{v \in V} \|G^{V;v}_0\|_{\ell^q(V)}.
\end{equation}
 The forward Born series (\ref{eq:f_born_r2dmap}) converges if
\begin{equation}
\mu_p \|\eta\|_p <1.
\end{equation}
Moreover, the $N$-term truncation error has the following bound,
\begin{equation}
\left\| \Lambda_\eta -\left(G_0+ \sum_{j=N}^\infty K_j (\eta,\cdots,\eta)\right) \right\|_{\ell^p(R\times S)}\le \nu_p\|\eta\|_p^{N+1} \mu_p^{N}\frac{1}{1-\mu_p \|\eta\|_p}.
\end{equation}
\end{proposition}

\begin{remark}
The bounds we obtain are similar to those found in the continuous setting \cite{ip2008}, though here we present a novel proof of $\ell^2$-boundedness and extend our results to include $p \in[1,2);$ a case not previously considered. 
\end{remark}

Before proving the proposition, we first establish the following useful identities.
\begin{lemma}
Let $M$ be an $n \times n$ matrix, and $D_{a},$ $D_{b}$ be $n \times n$ diagonal matrices with diagonal entries given by vectors $a$ and $b,$ respectively. Let $M(k)$ denote the $k$th row of $M,$ and 
\begin{equation}
C_{M,q} = \max_{k} \| M(k) \|_q,
\end{equation}
for $1\le q\le \infty.$ Then for any vectors $u^T$ and $v,$ and $p,q \in [1,\infty],$ such that $1/p +1/q = 1,$
\begin{equation}
|u^T D_{a} M D_{b} v| \le C_{M,q} \|u\|_q \|a\|_p \|b\|_p \| v\|_\infty.
\end{equation}
\end{lemma}
\begin{proof}
We begin by observing that if $e_k$ is the $k$th canonical basis vector, $a = \sum_{k} D_a e_k$ and
$I = \sum_{k} e_k e_k^T,$ where $I$ is the $n \times n$ identity matrix. Hence
\begin{equation}
\begin{split}
|u^T D_{a} M D_{b}v| & \le\left( \sum_{k} | u^T D_{a} e_k| \right) \max_{k} \left| M(k) D_{b} v \right|,\\
&\le \|u\|_q \|a\|_p \max_{k} \sum_{j} |M(k)\, D_{b}\, e_j| \max_j |{e_j^T v}|,\\
&\le \|u\|_q \|a\|_p \|b\|_p \max_k \|M(k)\|_q \max_j |{e_j^T v}|.
\end{split}
\end{equation}
\end{proof}
We can iterate the result of the Lemma to obtain the following corollary.

\begin{corollary}
Let $M_1,\cdots, M_{j-1}$ be $n \times n$ matrices and $D_{{a}_1}, \cdots, D_{{a}_j}$ be $n \times n$ diagonal matrices with diagonal elements given by the vectors $a_1, \cdots, a_j.$ If $M_i(k)$ and $C_{M_i,q}$ are defined as in the previous Lemma, then for all $u$ and $v,$
\begin{equation}
\left| u^ T D_{{a}_1} M_1 D_{{a}_2} \cdots M_{j-1} D_{{a}_j} v\right|\le \|a_1\|_p \cdots \|a_j\|_p\, C_{M_1,q} \cdots C_{M_{j-1},q}\|u\|_q \|v\|_\infty, 
\end{equation}
where once again $p,q \in[1,\infty]$ and $1/p+1/q = 1.$
\end{corollary}

We now return to the proof of Proposition \ref{thm_fwd_brn_thm}.
\begin{proof}
 Since $D_{{\eta}_i}$ is a diagonal matrix, $D_{{\eta}_i} = \sum_{k \in V} \eta_i(k) e_k e_k^T,$ where $\eta_i(k)$ is the $k$th component of the vector $\eta_i$ and $e_k$ is the canonical basis vector corresponding to the vertex $k.$ From the definition of $K_j$, we see that
 \begin{equation}
\begin{split}
\|K_j(\eta_1,\cdots,\eta_n) \|_{p} &\le \alpha_0^j \left(\sum_{r \in R, \, s \in S}  \left[G_0^{r;V} D_{{\eta}_1} \,G_0^{V;V} \,D_{{\eta}_2} \cdots G_0^{V;V} D_{{\eta}_j} G_0^{V;s}\right]^p\right)^{1/p}.\\
\end{split}
\end{equation}
The previous Corollary implies that
 \begin{equation}
 \begin{split}
  \left|G_0^{r;V} \,D_{{\eta}_1} \cdots D_{{\eta}_j} G_0^{V;s} \right| &\le \|\eta_1\|_p \cdots \|\eta_{j}\|_p C_{G_0^{V;V},q}^{j-1} \|G_0^{V,r}\|_q \|G_0^{V,s}\|_\infty.\\
  \end{split}
 \end{equation}
Thus
  \begin{equation}
\begin{split}
\|K_j \|_{p} & \le \alpha_0^j \|G_0^{R,V}\|_{\ell^p(R) \times \ell^q(V)} \|G_0^{V,S}\|_{\ell^q(V) \times \ell^p(S)} C_{G_0^{V;V},q}^{j-1}\,,\\
&\le \nu_p \, \mu_p^{j-1},
\end{split}
\end{equation}
where $\nu_p =\alpha_0 \|G_0^{R,V}\|_{\ell^q(V) \times \ell^p(R)} \|G_0^{V,S}\|_{\ell^q(V) \times \ell^p(S)}$ and $\mu_p = \alpha_0\, C_{G_0^{V;V},q},$ from which the result follows immediately.

\end{proof}
 
\subsection{Inverse Born series}\label{subsec_inv_brn}

Proceeding as in \cite{ip2008}, let $\phi \in \ell^2 (R \times S)$ denote the {\it scattering data},
\begin{equation}
\phi(r,s) = G_0(r,s)-\Lambda_\eta(r,s),
\end{equation}
corresponding to the difference between the measurements in the background medium and those in the medium with the potential present. Note that if the forward Born series converges, we have
\begin{equation}
\label{born_again}
\phi(r,s) = \sum_{j=1}^\infty K_j(\eta,\dots,\eta).
\end{equation}
Next, we introduce the ansatz 
\begin{equation}
\label{eq_eta_ansatz}
\eta= \mcalK_1 (\phi) +\mcalK_2( \phi , \phi)+\mcalK_3( \phi , \phi, \phi) +\cdots \ ,
\end{equation}
where each $\mcalK_n$ is a multilinear operator. Though $\phi$ can be thought of as an operator from $\ell^2(R)$ to $\ell^2(S),$ in (\ref{eq_eta_ansatz}) we treat it as a vector of length $|R| \cdot |S|.$ Similarly, though it is often convenient to think of $\eta$ as a (diagonal) matrix, in (\ref{eq_eta_ansatz}) it should be thought of as a vector of length $|V|.$ Treating $\eta$ and $\phi$ as matrices results in a different inverse problem related to matrix completion~\cite{vadim_mark}. With a slight abuse of notation, we also use $K_1$ to denote the $|R||S| \times |V|$ matrix mapping $\eta$ (viewed as a vector) to $K_1\eta$, once again thought of as a vector.

To derive the inverse Born series, we substitute the ansatz (\ref{eq_eta_ansatz}) into the forward series (\ref{born_again}) and equate tensor powers of $\phi$. We thus obtain which the following recursive expressions for the operators $\mcalK_j$~\cite{ip2008}: 
\begin{equation}\label{eq:def_calK}
\begin{split}
\mcalK_1 &= K^+_1,\\
\mcalK_2 &= -\mcalK_1 K_2 \mcalK_1 \otimes \mcalK_1,\\
\mcalK_3 &= -\left(\mcalK_2 K_1 \otimes K_2+\mcalK_2 K_2 \otimes K_1+\mcalK_1 K_3 \right) \mcalK_1 \otimes \mcalK_1 \otimes \mcalK_1,\\
\mcalK_j &= -\left( \sum_{m=1}^{j-1} \mcalK_m \sum_{i_1+\cdots+i_m=j} K_{i_1}\otimes \cdots \otimes K_{i_m} \right) \mcalK_1 \otimes \cdots \otimes \mcalK_1,
\end{split}
\end{equation}
where $K_1^+$ denotes the (regularized) pseudoinverse of $K_1$.

The following result provides sufficient conditions for the convergence of the inverse Born series for graphs where $|V| = |R\times S|$, corresponding to the case of a formally determined inverse problem.

\begin{theorem}\label{thm_complex}
Let $|V| = |R\times S|$ and $p \in [1,\infty]$. Suppose that the operator $K_1$ is invertible. Then the inverse Born series converges to the original potential, $\eta,$ if $\|\phi\|_p <r_p$. Here the radius of convergence $r_p$ is defined by
\begin{equation}
\label{def_rp}
r_p = \frac{C_p}{\mu_p}\left[1-2\frac{\nu_p}{C_p}\left(\sqrt{1+\frac{C_p}{\nu_p}}-1 \right) \right] ,
\end{equation}
where
\begin{equation}
C_p = \min_{\|\eta\|_p=1} \|K_1(\eta)\|_p 
\end{equation}
and $\nu_p, \mu_p$ are defined in (\ref{eq:forward_constants_def}).	
\end{theorem}

The proof of Theorem~\ref{thm_complex} requires the following multi-dimensional version of Rouch\'{e}'s theorem.
\begin{theorem}~[Theorem 2.5, \citealp{several_complex}]
\label{multi_rouche}
Let $D$ be a domain in $\mathbb{C}^n$ with a piecewise smooth boundary $\partial D$.
Suppose that $f,g:\mathbb{C}^n \rightarrow \mathbb{C}^n$ are holomorphic on $\bar{D}$. If for each point ${z} \in \partial D$ there is at least one index $j$, $j=1,\dots,n$, such that $|g_j({z})| < |f_j({z})|$, then $f({z})$ and $f({z})+g({z})$ have the same number of zeros in $D,$ counting multiplicity.
\end{theorem}

\begin{proof}[Proof of Theorem \ref{thm_complex}]
Put $n=|V| = |R\times S|$. Let $F:\mathbb{C}^n\times\mathbb{C}^n \rightarrow \mathbb{C}^n$ be the function defined by
\begin{equation}
F(\eta,\phi) = \phi - \sum_{j=1}^\infty K_j(\eta,\dots,\eta).
\end{equation}
Note that $F$ has $n$ components $F_1,\dots,F_n,$ each of which is well-defined and holomorphic for all $\phi$ if $\|\eta\|_p < 1/\mu_p.$ Let
\begin{equation}
C_p = \min_{\|\eta\|_p=1} \|K_1(\eta)\|_p,
\end{equation}
 which is non-zero for all $p$ since $K_1$ is invertible. Then
\begin{equation}
\begin{split}
\|F(\eta,0)\|_p &\ge C_p \|\eta\|_p -\sum_{j=2}^\infty \|K_j(\eta,\dots,\eta)\|_p,\\
& \ge C_p \|\eta\|_p - \nu_p \sum_{j=2}^\infty \mu_p^{j-1} \|\eta\|_p^j,\\
& \ge C_p \|\eta\|_p - \nu_p \mu_p \|\eta\|_p^2 \frac{1}{1-\mu_p\|\eta\|_p},
\end{split}
\end{equation}
where the second inequality follows from the bounds on the forward operators obtained in the proof of Proposition \ref{thm_fwd_brn_thm}. For $0<\|\eta\|_p < 1/\mu_p,$ $\|F(\eta,0)\|_p$ is non-vanishing if
\begin{equation}
\|\eta\|_p < \frac{1}{\mu_p} \frac{C_p}{C_p+\nu_p}.
\end{equation}
Suppose $\lambda \ge 1$. We then define
\begin{equation}
R_\lambda = \frac{1}{\mu_p} \frac{C_p}{C_p+\nu_p \lambda},
\end{equation}
and let $\Omega_{1,\lambda} = \{ \eta \in \mathbb{C}^n\,|\, \|\eta\|_p <R_\lambda\}.$ 

Next, we observe that $F(\eta,\phi)-F(\eta,0) = \phi$ and hence if
\begin{equation}\label{eq:conv_proof_cond}
\|\phi\|_p < \|F(\eta,0)\|_p,
\end{equation}
then
\begin{equation}
\|F(\eta,\phi)-F(\eta,0)\|_p < \|F(\eta,0)\|_p.
\end{equation}
Note that
\begin{equation}
\| F(\eta,0) \|_p \ge C_p \|\eta\|_p -\nu_p\mu_p \frac{\|\eta\|_p^2}{1-\mu_p\| \eta\|_p},
\end{equation}
and thus (\ref{eq:conv_proof_cond}) holds if
\begin{equation}
\|\phi\|_p < C_p \|\eta\|_p -\nu_p\mu_p \frac{\|\eta\|_p^2}{1-\mu_p\| \eta\|_p}.
\end{equation}
If $\eta \in \partial \Omega_{1,\lambda},$ (\ref{eq:conv_proof_cond}) holds if
\begin{equation}
\|\phi\|_p < R_\lambda C_p \left(1 - \frac{1}{\lambda} \right)\equiv r_{p,\lambda}.
\end{equation}
Defining $\Omega_{2,\lambda} = \{\phi \in \mathbb{C}^n\,|\, \|\phi\|_p < r_{p,\lambda}\},$ we note the following: for all $(\eta,\phi) \in \Omega_{1,\lambda}\times\Omega_{2,\lambda},$ $F(\eta,0) \neq \vec{0};$ and, for all $(\eta,\phi) \in \partial\Omega_{1,\lambda}\times\Omega_{2,\lambda},$ $\|F(\eta,\phi)-F(\eta,0)\|_p < \|F(\eta,0)\|_p.$ By Theorem \ref{multi_rouche}, $F(\eta,0)$ and $F(\eta,\phi)$ have the same number of zeroes counting multiplicity on $\Omega_{1,\lambda}\times\Omega_{2,\lambda}$, namely precisely one. Thus, for all $\phi \in \Omega_{2,\lambda}$  there exists a unique $\eta = \psi(\phi)$ such that $F(\psi(\phi),\phi) = 0.$ Since the unique zero must have multiplicity one,
\begin{equation}
{\rm det}\left( \{\partial_{\eta_j}F_i\,(\psi(\phi),\phi)\}_{i,j=1}^n\right) \neq 0.
\end{equation}
Consequently, by the analytic implicit function theorem~[Theorem 3.1.3, \citealp{sev_comp}], $\psi$ is analytic in a neighborhood of each $\phi \in \Omega_{2,\lambda}$, which is sufficient to prove that $\psi$ is analytic on all of $\Omega_{2,\lambda}$. Hence $\psi$ has a Taylor series converging absolutely for all $\phi \in \Omega_{2,\lambda}.$ By construction, the terms in this series must match those of the inverse Born series. It follows that the inverse Born series must also converge for all $\phi \in \Omega_{2,\lambda}.$ Optimizing over $\lambda \ge 1,$ the inverse Born series converges for all $\phi\in \mathbb{C}^n,$ such that
\begin{equation}\label{eq:phi_bound}
\|\phi\|_p < \frac{C_p}{\mu_p}\left[1-2\frac{\nu_p}{C_p}\left(\sqrt{1+\frac{C_p}{\nu_p}}-1 \right) \right],
\end{equation}
which completes the proof.
\end{proof}

\begin{remark}

We note that Theorem 6 is closely related to the problem of determining the domain of biholomorphy of a function of several complex variables, where the  radii of analyticity of the function and its inverse are referred to as {Bloch radii} or {Bloch constants} \cite{fixed-point,harris,bloch_chen}. In the context of nonlinear optimization a related result was obtained in \cite{rouche_one}, which also made use of Rouch\'{e}'s theorem.

\end{remark}

\begin{remark}
The bound constructed in Theorem \ref{thm_complex} is only a lower bound for the radius of convergence. In practice, the series converges well outside this range, as the example in the next section confirms. Additionally, if in the proof of Theorem \ref{thm_complex} we instead define $F(\eta,\phi)$ by
\begin{equation}
F(\eta,\phi) = \mathcal{K}_1\phi - \sum_{j=1}^\infty \mathcal{K}_1 K_j(\phi,\dots,\phi),
\end{equation}
then it can easily be shown that the inverse series converges if
\begin{equation}\label{eq:mathcalk1_bound}
\|\mathcal{K}_1 \phi\|_p <\frac{1}{\mu_p}\left[1-2\frac{\nu_p}{C_p}\left(\sqrt{1+\frac{C_p}{\nu_p}}-1 \right) \right].
\end{equation}
Though the right-hand side is slightly more complicated, it is often easily computed and gives a better bound.
\end{remark}

Figure \ref{fig:asymptotics} shows a plot of the bound on the radius of convergence,
$$r_p=\frac{C_p}{\mu_p}\left[1-2\frac{\nu_p}{C_p}\left(\sqrt{1+\frac{C_p}{\nu_p}}-1 \right) \right]$$
 for various values of $C_p/\nu_p.$ For large graphs we expect the determinant of $K_1$ to be small, corresponding to a small value of $C_p.$ In this regime we observe that the first term in the asymptotic expansion of (\ref{eq:phi_bound}) is
\begin{equation}\label{eq:asymp_r_p}
r_p = \frac{C_p^2}{4\nu_p\mu_p} + \mathrm{O}\left(C_p^3\right).
\end{equation}
 
\begin{figure}[t]
  \centering
    \includegraphics[width=0.6\textwidth]{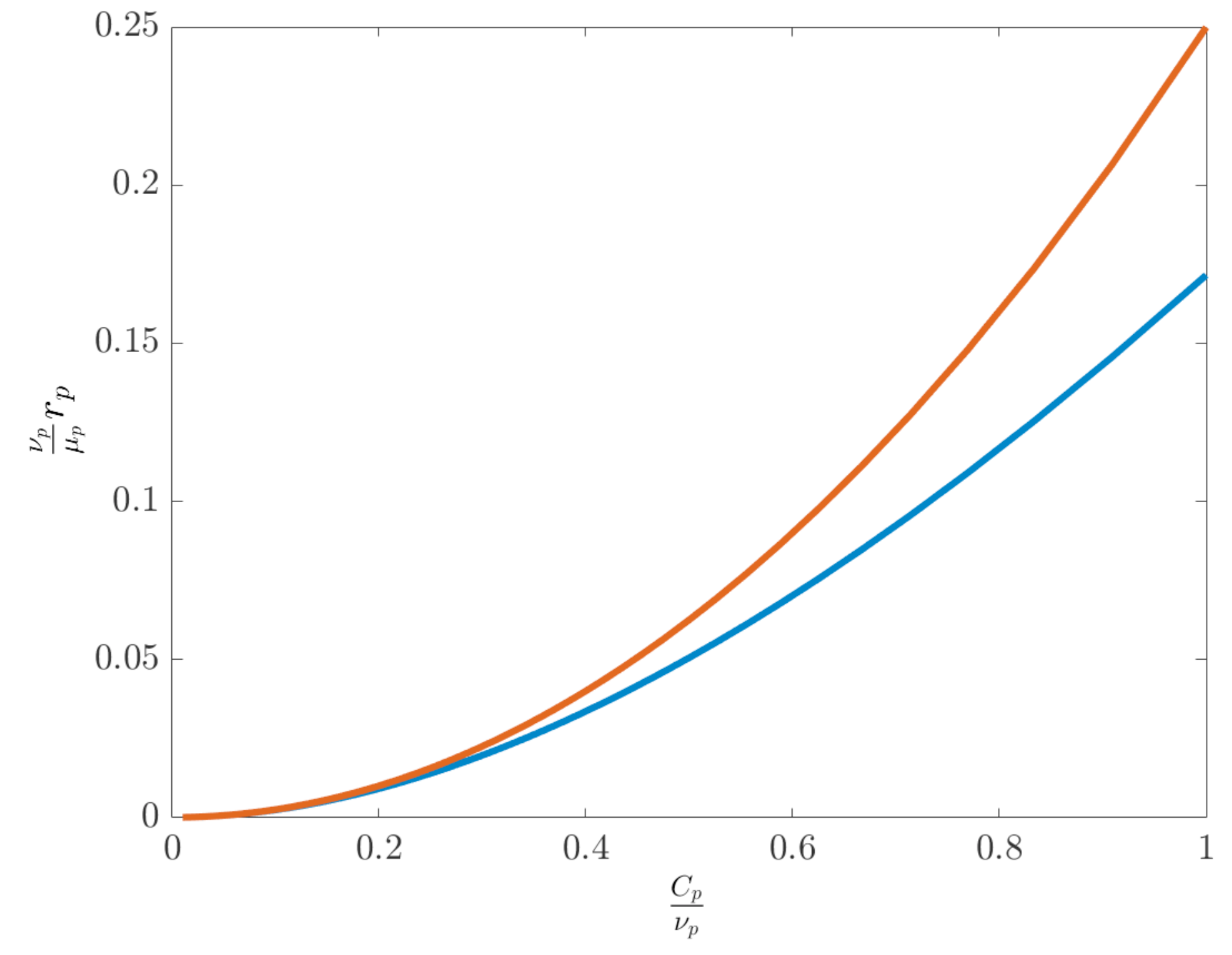}
      \caption{The bound on the radius of convergence of the inverse Born series as a function of $C_p/\nu_p.$ The radius $r_p$ (multiplied by $\nu_p/\mu_p$) is shown in blue. The red curve is the asymptotic estimate given in (\ref{eq:asymp_r_p}). } \label{fig:asymptotics}
\end{figure}

We now consider the stability of the limit of the inverse scattering series under perturbations in the scattering data. The following stability estimate follows immediately from Theorem \ref{thm_complex}.

\begin{proposition}
Let $D$ be a compact subset of 
$\Omega_p =  \left\{ \phi \in \mathbb{C}^n\,|\, \|\phi\|_p <r_p \right\},$
where $r_p$ is defined in (\ref{def_rp}) and $p\in [1,\infty]$. Let $\phi_1$
and $\phi_2$ be scattering data belonging to $D$ and $\psi_1$
and $\psi_2$ denote the corresponding limits of the inverse Born series.
Then the following stability estimate holds:
\begin{equation*}
\|\psi_1 - \psi_2\|_p \le M \| \phi_1 - \phi_2\|_p \ ,
\end{equation*}
where $M=M(D,p)$ is a constant which is otherwise independent of $\phi_1$ and $\phi_2$.
\end{proposition}
\begin{proof}
In the proof of Theorem \ref{thm_complex} it was shown that $\psi$ is analytic on $\Omega_p.$ In particular, it follows that there exists an $M<\infty$ such that 
\begin{equation}
\| D\psi \|_p \le M, 
\end{equation}
for all $\phi \in \Omega.$ Here $D \psi$ is the differential of $\psi$ and $\| \cdot\, \|_p$ is its induced matrix $p$-norm. By the mean value theorem,
\begin{equation}
\|\psi_1-\psi_2\|_p \le M \|\phi_1 - \phi_2\|_p,
\end{equation}
for all $\phi_1, \phi_2 \in \Omega.$
\end{proof}

Theorem~\ref{thm_complex} guarantees convergence of the inverse Born series, but does not provide an estimate of the approximation error. Such an estimate is provided in the next theorem.

\begin{theorem}
Suppose that the hypotheses of Theorem~\ref{thm_complex} hold and $\|\phi\|_p < \tau r_p,$ where $\tau <1$. If $\eta$ is the true vertex potential corresponding to the scattering data $\phi$, then
$$\left\| \eta - \sum_{m=1}^N K_m(\phi,\dots,\phi) \right\|_\infty <  M\left(\frac{1}{1-\tau}\right)^n \left(\frac{\|\phi\|_p}{\tau r_p} \right)^N \frac{1}{1-\frac{\|\phi\|_p}{\tau r_p}}.$$
\end{theorem}
\begin{proof}
The proof follows a similar argument as the one used to show uniform convergence of analytic functions on polydiscs, see [Lemma 1.5.8 and Corollary 1.5.9, \citealp{sev_comp}] for example. By Theorem~\ref{thm_complex}, since $\|\phi \|_p < r_p$, the inverse Born series converges. Moreover, the value to which it converges is precisely the unique potential ${\eta}$ corresponding to the scattering data $\phi.$

Let $\psi_j$ be the $j$th component of the sum of the inverse Born series, which is of the form
\begin{equation}\label{eq:inv_born_psi}
\psi_j = \sum_{|\alpha|=0}^\infty c_\alpha^{(j)} \phi^\alpha ,
\end{equation}
for suitable $c_\alpha^{(j)}$, consistent with (\ref{eq_eta_ansatz}).
Here we have used the following notational convention: if $\alpha=(\alpha_1,\dots,\alpha_n)$ then $\phi^\alpha \equiv \phi_1^{\alpha_1}\dots \phi_n^{\alpha_n}.$ Additionally, for a given multi-index $\alpha$ we define $|\alpha| = \alpha_1+\dots+\alpha_n.$ Note that each $\alpha$ in the sum has exactly $n$ elements, though any number of them may be zero.

Let 
$$\psi_j^{(N)} =  \sum_{|\alpha|=0}^N c_\alpha^{(j)} \phi^\alpha,
$$
and $\Delta_\phi$ be the polydisc
$$\Delta_\phi = \left\{ z\in \mathbb{C}^n\,|\,|z_s| < |\phi_s|  \frac{r_p}{\|\phi\|_p}, \,s=1,\dots,n \right\}.$$
We note that $\phi \in \Delta_\phi \subseteq \{\phi\,|\, \|\phi\|_p < r_p\}.$ It follows by Cauchy's estimate~[Theorem 1.3.3, \citealp{sev_comp}] that
\begin{equation}
|c_\alpha^{(j)}| \le M \left(\frac{\|\phi\|_p}{r_p}\right)^{|\alpha|} \frac{1}{|\phi|^\alpha},
\end{equation}
where $M = \max_{\|\phi\|_p <r_p} \|\psi\|_p.$ To proceed, we employ the following combinatorial identity, ~[Example 1.5.7, \citealp{sev_comp}],
\begin{equation}
\sum_{|\alpha|=0}^\infty t^{|\alpha|} = \frac{1}{(1-t)^n} , 
\end{equation}
for all $ t\in (-1,1).$ In light of the above, we see that

$$M \sum_{|\alpha| = 0} \left(\frac{\|\phi\|_p}{r_p} \right)^{|\alpha|}= M \left(\sum_{m=1}^\infty \left(\frac{\|\phi\|_p}{r_p} \right)^{m}\right)^n= M \left(\frac{1}{1-\frac{\|\phi\|_p}{r_p} } \right)^n.$$
The function $1/(1-t)^n$ is bounded by
$$M\left(\frac{1}{1-\tau}\right)^n $$
for all $|t|<\tau<1.$ Thus the one-dimensional Cauchy estimate implies that the $k$th coefficient of its Taylor series, $b_k,$ is bounded by
$$|b_k| \le M\left(\frac{1}{1-\tau}\right)^n \frac{1}{\tau^k}, $$
and so
\begin{equation}
\begin{split}
\sum_{|\alpha|>N} \left(\frac{\|\phi\|_p}{r_p} \right)^{|\alpha|} &\le M\sum_{k>N}^\infty \left(\frac{1}{1-\tau}\right)^n \left(\frac{\|\phi\|_p}{\tau r_p} \right)^k,\\
& = M\left(\frac{1}{1-\tau}\right)^n \left(\frac{\|\phi\|_p}{\tau r_p} \right)^N \frac{1}{1-\frac{\|\phi\|_p}{\tau r_p}}.
\end{split}
\end{equation}
Hence, independent of $j,$
\begin{equation}
\begin{split}
\left|\psi_j - \psi_j^{(N)}\right| &= \left| \sum_{|\alpha|>{N}} c_\alpha^{(j)} \phi^\alpha \right|,\\
& \le \sum_{|\alpha|>{N}} |c_\alpha^{(j)}| |\phi|^\alpha,\\
& \le M \sum_{|\alpha|>N} \left( \frac{\|\phi\|_p}{r_p} \right)^{|\alpha|},\\
& \le M\left(\frac{1}{1-\tau}\right)^n \left(\frac{\|\phi\|_p}{\tau r_p} \right)^N \frac{1}{1-\frac{\|\phi\|_p}{\tau r_p}},
\end{split}
\end{equation}
from which the result follows immediately.
\end{proof}

\begin{remark}
Note that in the previous theorem we can minimize our bound over $\tau\in (\|\phi\|_p/r_p,1).$ Letting $\gamma = \|\phi\|_p/r_p$ the minimum occurs at
$$\tau_* = \frac{\gamma}{2}\left[\left(1+ \frac{N-\gamma}{\gamma(n+N)}\right)+\sqrt{\left(1-\frac{N-\gamma}{\gamma(N+n)}\right)^2+4\frac{1-\gamma}{\gamma(N+n)}}\right].$$
\end{remark}

Finally, we conclude our discussion of the convergence of the inverse Born series by proving an asymptotic estimate for the truncation error. Specifically, we show that for a fixed number of terms $N$ the error in the $N$-term inverse Born series goes to zero as $\eta$ goes to zero. We note that our estimate does not apply to the case of fixed $\phi$ and $N \rightarrow \infty$ since $C_{N,a} x^N \rightarrow \infty$ for any fixed positive $x.$

\begin{theorem}\label{thm_asymp_conv}
Let $\|\eta\|_p \mu_p <a<1$. Then there exists a constant $C_{N,a},$ depending on $N$ such that
\begin{equation}
\left\| \eta - \sum_{j=1}^N \mcalK_j(\phi,\cdots,\phi)\right\|_p \le C_{N,a}  \|\eta\|_p^{N+1}.
\end{equation}
\end{theorem}
\begin{proof}
We begin by considering the truncated inverse Born series,
\begin{equation}\label{eq:trunceta}
\eta_N (\phi) = \sum_{j=1}^N \mcalK_j(\phi,\cdots,\phi),
\end{equation}
noting that
\begin{equation}
\eta_N - \eta = \sum_{j=1}^N \mcalK_j (\phi,\cdots,\phi) - \eta.
\end{equation}
If $\mu_p \|\eta\|_p < 1,$ $\phi$ is equal to its forward Born series, and hence
\begin{equation}
\eta_N - \eta = \sum_{j=1}^N\,\sum_{i_1,\cdots,i_j=1}^\infty \mcalK_j[K_{i_1}(\eta),\cdots,K_{i_N}(\eta)]-\eta.
\end{equation}
Using (\ref{eq:def_calK}) we find that
\begin{equation}
\eta_n - \eta = \sum_{j=1}^n\,\sum_{i_1+\cdots+i_j>n}^\infty \mcalK_j[K_{i_1}(\eta),\cdots,K_{i_N}(\eta)],
\end{equation}
which follows from the construction of the inverse Born series. Therefore
\begin{equation}\label{eq:inv_eta_bnd}
\begin{split}
\|\eta_N - \eta\|_p &\le \sum_{j=1}^N \|\mcalK_j\|_p\, \nu_p^j \sum_{k>N}^\infty \mu_p^{k-j} \|\eta\|_p^k,\\
&le \sum_{j=1}^N \|\mcalK_j\|_p\, \left(\frac{\nu_p}{\mu_p}\right)^j \sum_{k>N}^\infty \mu_p^{k} \|\eta\|_p^k,\\
&le \sum_{j=1}^N \|\mcalK_j\|_p\, \left(\frac{\nu_p}{\mu_p}\right)^j \|\eta\|_p^{N+1} \mu_p^{N+1} \sum_{k=0}^\infty (\mu_p \|\eta\|_p)^k,\\
& = \sum_{j=1}^N \|\mcalK_j \|_p\, \nu_p^j \mu_p^{N+1-j} \|\eta\|_p^{N+1} \frac{1}{1- \mu_p \|\eta\|_p}.
\end{split}
\end{equation}

In order to proceed, we require a bound on $\|\mcalK_j\|_p.$ As in \cite{ip2008}, we begin by observing that if $p \in[1,\infty],$ $j>2,$
\begin{equation}
\begin{split}
\|\mcalK_j\|_p &\le \|\mcalK_1 \|_p^j \sum_{m=1}^{j-1} \| \mcalK_m\|_p \sum_{i_1+\cdots+i_m=j} \|K_{i_1}\|_p \cdots \|K_{i_m}\|_p,\\
&\le \|\mcalK_1 \|_p^j \sum_{m=1}^{j-1} \| \mcalK_m\|_p \sum_{i_1+\cdots+i_m=j} \left(\frac{\nu_p}{\mu_p}\right)^m \mu_p^{j},\\
&= \|\mcalK_1 \|_p^j \sum_{m=1}^{j-1} \| \mcalK_m\|_p {{j-1}\choose{m-1}} \left(\frac{\nu_p}{\mu_p}\right)^m \mu_p^{j},\\
&\le \nu_p \mu_p^{j-1} \|\mcalK_1 \|_p^j \left(\sum_{m=1}^{j-1} \| \mcalK_m\|_p \right) \sum_{m=0}^{j-2} {{j-1}\choose{m}}\left(\frac{\nu_p}{\mu_p}\right)^m,\\
\end{split}
\end{equation}
where we have shifted the index $m$ in the last expression. It follows immediately from the binomial thoerem that
\begin{equation}
\begin{split}
\|\mcalK_j\|_p &\le \|\mcalK_1 \|_p^j \nu_p \left[\left( \mu_p + \nu_p\right)^{j-1}-\nu_p^{j-1} \right]\left(\sum_{m=1}^{j-1} \| \mcalK_m\|_p \right),\\
&\le \left[ \|\mcalK_1 \|_p \left( \mu_p + \nu_p\right) \right]^j\left(\sum_{m=1}^{j-1} \| \mcalK_m\|_p \right),\\
&\le \|\mcalK_1 \|_p \left( \mu_p + \nu_p\right) \|\mcalK_{j-1}\|_p+\frac{\nu_p}{\mu_p+\nu_p} \left[ \|\mcalK_1 \|_p \left( \mu_p + \nu_p\right) \right]^j \|\mcalK_{j-1}\|_p,\\
&\le \|\mcalK_1\|_p (\mu_p+\nu_p)\left[1+(\mu_p+\nu_p)^{j-1} \|\mcalK_1\|_p^{j-1}\right] \|\mcalK_{j-1}\|_p.
\end{split}
\end{equation}
Further note that if $j=2,$ then
\begin{equation}
\|\mcalK_2\|_p \le \|\mcalK_1\|_p^3 \nu_p^2 \le \| \mcalK_1\|_p^3 (\mu_p + \nu_p)^2.
\end{equation}
For ease of notation, let $r=(\mu_p+\nu_p) \|\mcalK_1\|_p$ and note that
\begin{equation}\label{eq:mcalK_bnd}
\begin{split}
\|\mcalK_j\|_p &\le \|\mcalK_1\|_p (\mu_p+\nu_p)\left[1+(\mu_p+\nu_p)^{j-1} \|\mcalK_1\|_p^{j-1}\right] \|\mcalK_{j-1}\|_p,\\ 
& \le r \, [1+r^{j-1}] \|\mcalK_{j-1} \|_p,\\
&\le \|\mcalK_1\|_p r^j 2^{j} \left[\max\{ 1,r \}\right]^\frac{j(j-1)}{2}.
\end{split}
\end{equation}
If we define $C = \max\{1,r\},$ then it follows from (\ref{eq:inv_eta_bnd}) and (\ref{eq:mcalK_bnd})
\begin{equation}
\begin{split}
\| \eta_N - \eta\|_p  &\le \frac{\|\eta\|_p^{N+1}\mu_p^{N+1}\|\mcalK_1\|_p}{1-\mu_p \|\eta\|_p} \sum_{j=1}^N \left( \frac{2\nu_p\,r}{\mu_p}\right)^j C^\frac{j^2}{2},\\
& \le \frac{\|\eta\|_p^{N+1}\mu_p^{N+1}\|\mcalK_1\|_p}{1-\mu_p \|\eta\|_p} \frac{1- \left(2\nu_p \mu_p^{-1}r \right)^{N+1}}{1- 2\nu_p \mu_p^{-1} r} C^\frac{N^2}{2},\\
&\le \tilde{C}(N)\frac{ \|\eta\|_p^{N+1}}{1-\mu_p \|\eta\|_p}.
\end{split}
\end{equation}
Thus, for $\|\eta\|_p< \mu_p^{-1} a< \mu_p^{-1},$ 
\begin{equation}
\|\eta_N-\eta\|_p \le C'(N) \|\eta\|_p^{N+1}
\end{equation}
for some constant $C'(N).$
\end{proof}

\section{Implementation}\label{sec:imp}
\subsection{Regularizing $\mcalK_1$}
In the previous section we found that the norm of $\mcalK_1$  plays an essential role in controlling the convergence of the inverse Born series. In practice, for large graphs $\|\mcalK_1\|_p$ is too large to guarantee convergence of the inverse series. Moreover, even if the series converges, a modest amount of noise can lead to large changes in the recovered potential. Regularization improves the stability and radius of convergence of the inverse Born series by replacing $\mcalK_1$ by a more stable operator $\tilde{\mcalK}_1$. In our implementations we use Tikhonov regularization~\cite{natterer}.

\subsection{Numerical examples}
We consider a $12 \times 12$ lattice, with all boundary vertices acting as sources and receivers. We simulate the scattering data by solving the forward problem for a particular vertex potential. As is often the case in biomedical applications, we consider a homogeneous medium with a small number of large inclusions. Figure~\ref{fig:ex_recon} shows a typical result of the inverse Born series method. In this example, $\mu_2 \sim 0.1738,$ $\nu_2\sim10.47,$ and $\|\mathcal{K}_1\|_2 \sim 2.9 \times10^{8},$ which gives a radius of convergence of $\|\mathcal{K}_1 \phi\|_2 < 4.5 \times10^{-9}$ using (\ref{eq:mathcalk1_bound}), or $\|\phi\|_2 < 1.51 \times10^{-16}$ for (\ref{eq:phi_bound}).
For $\|\eta\|_\infty = 0.1,$ $\|\phi\|_2 \sim 0.1853$ and $\|\mathcal{K}_1 \phi\|_2 \sim 16.2522,$ while for $\|\eta\|_\infty = 0.5,$ $\|\phi\|_2 \sim 0.9012$ and $\|\mathcal{K}_1 \phi\|_2 \sim 393.$ Therefore, though both of these $\phi$'s lie outside our bound, the first example shows signs of convergence, while the latter one appears to diverge. 

\renewcommand{\thesubfigure}{\roman{subfigure}}
\begin{figure}
    \centering
    \iffalse
    \begin{subfigure}[b]{0.4\textwidth}
        \includegraphics[width=\textwidth]{eta_examp_0_01}
        \caption{}
    \end{subfigure}
    \hspace{1 cm}
    \fi
    ~
    \begin{subfigure}[b]{0.6\textwidth}
        \includegraphics[width=\textwidth,height = 8.0 cm]{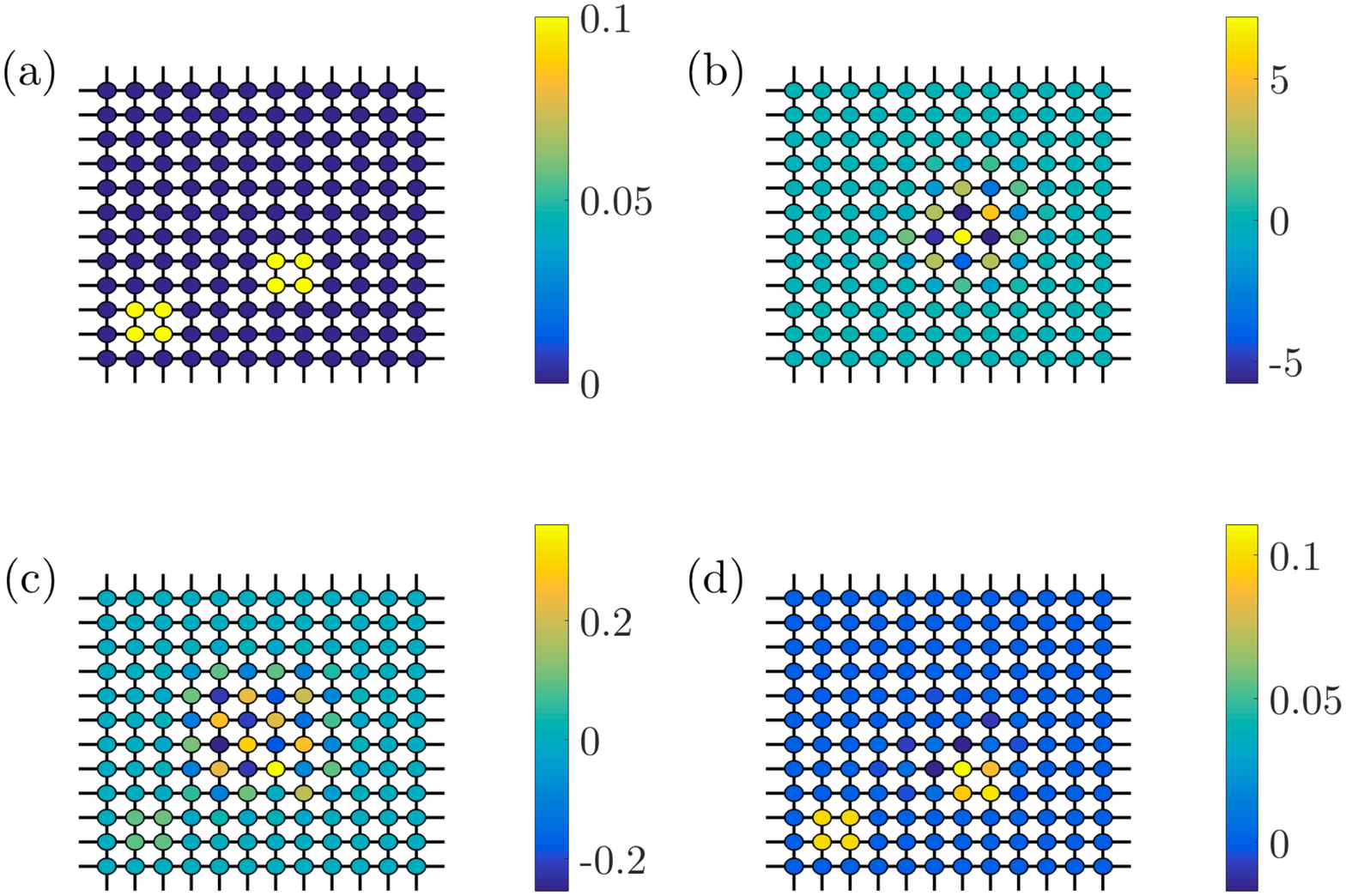}
        \caption{}
    \end{subfigure}\newline
      ~ 
    \begin{subfigure}[b]{0.7\textwidth}
        \includegraphics[width=\textwidth]{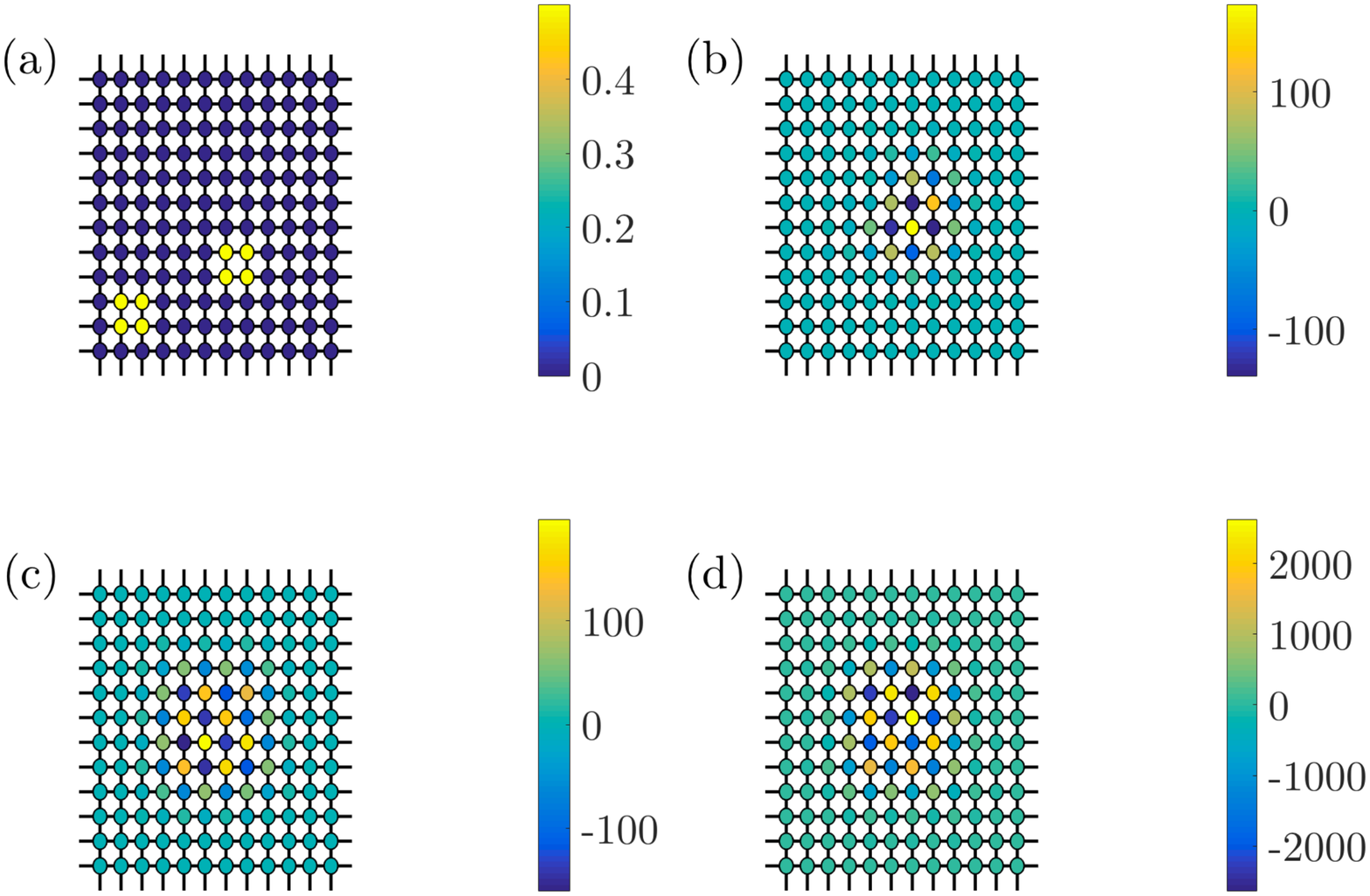}
        \caption{}
    \end{subfigure}
    \caption{Inverse Born numerical experiments for a $12 \times 12$ lattice with $\alpha_0 = 0.1,$ $t = 0$ and receivers and sources placed at each vertex on the boundary. i) $\|\eta\|_\infty = 0.1,$ and ii) $\|\eta\|_\infty = 0.5.$ In each group, a) is the absorption, $\eta,$ used to generate measurement data; b) is first term of the inverse Born series; c) is the first two terms of the inverse Born series; and d) is the first 5 terms of the inverse Born series.}\label{fig:ex_recon}
\end{figure}

\section{Incorporating potential structure}\label{sec:mult_freq}
The inverse Born series algorithm can be extended to take into account additional constraints on the vertex potential $\eta,$ such as restrictions on its support or requirements that it is constant on some subset of the domain, allowing the recovery of vertex potentials which would otherwise be unrecoverable using the inverse Born series described above.

\begin{theorem}\label{thm_mod_inv_brn}
Let $F$ be a linear mapping from $\mathbb{R}^k \rightarrow \mathbb{R}^{|V|},$ where $k \le |V|$ and suppose that $\eta$ is in the image of $F.$ Let $\eta'$ be its pre-image,
\begin{equation}\label{eq:mod_F_def}
\eta' = F^{-1} (\eta). 
\end{equation}a
Then
\begin{equation}
\eta' = \mcalK'_1 (\phi)+\mcalK'_2 \,(\phi ,\phi)+ \dots + \mcalK'_n (\phi ,\dots,\phi )+ \dots,
\end{equation}
and
\begin{equation}\label{eq:mod_inv_op_n}
\begin{split}
\mcalK'_1 &= (K_1\circ F)^+,\\
\mcalK'_2 &=-\mcalK'_1 \circ K_2 \circ \left( (F \circ \mcalK_1) \otimes (F \circ \mcalK_1)\right),\\
&\dots\\
\mcalK_n'&=-\sum_{j=1}^{n-1} \mcalK'_j \circ \left(\sum_{i_1+\dots i_j = n} K_{i_1}\otimes K_{i_2} \otimes \dots \otimes K_{i_j} \right)\circ\left( (F\circ \mcalK_1) \otimes \dots \otimes (F\circ \mcalK_1) \right),
\end{split}
\end{equation}
where $(K_1 \circ F)^+$ denotes the (regularized) pseudoinverse of $(K_1 \circ F).$
\end{theorem}
\begin{proof}
We begin by rewriting the discrete time-independent diffusion equation as
\begin{equation}\label{eq:sys_mod_1}
\begin{split}
L u + \alpha_0 [I+ D_{F(\eta')}] u -A_{V,\delta V}^Tv = \bf{0},\\
-A_{V,\delta V}u+Dv = \bf{g},
\end{split}
\end{equation}
where $D_{{F}(\eta')}$ is the diagonal matrix whose diagonal elements are given by the vector $F(\eta').$ If $\Lambda'_{\eta'}:\ell^2(\mathbb{R}^k) \rightarrow \ell^2(R\times S)$ denotes the Robin-to-Dirichlet map for the modified system (\ref{eq:sys_mod_1}) and $\eta$ is in the image of $F,$ then
\begin{equation}
 \Lambda'_{\eta'} = \Lambda_{\eta}.
\end{equation}
Thus the forward Born series of (\ref{eq:sys_mod_1}) is given by
\begin{equation}
\Lambda'_{\eta'}(s,r) = G_0(r,s) - \sum_{n=1}^\infty K_n \left(F(\eta'),\dots,F(\eta') \right).
\end{equation}
Following the construction of the inverse Born series, we let $\phi$ represent the measured data, and consider the ansatz
\begin{equation}
\eta' = \mcalK'_1( \phi)+\mcalK'_2 \,(\phi , \phi)+ \dots + \mcalK'_n (\phi, \dots, \phi) + \dots
\end{equation}
We see immediately that
\begin{equation}
\begin{split}
&\mcalK'_1 \circ K_1 \circ F = I,\\
&\mcalK'_1 \circ K_2\circ( F \otimes F)+\mcalK'_2 \circ \left((K_1\circ F) \otimes (K_1 \circ F)\right) = 0,\\
&\dots\\
&\sum_{j=1}^{n} \mcalK'_j \circ \left(\sum_{i_1+\dots i_j = n} K_{i_1}\otimes K_{i_2} \otimes \dots \otimes K_{i_j} \right)\circ\left( F \otimes \dots \otimes F \right)=0.
\end{split}
\end{equation}
If  $(K_1 \circ F)^+$ denotes the (regularized) pseudoinverse of $(K_1 \circ F),$ then we obtain
\begin{equation}\label{eq:mod_inv_op}
\begin{split}
\mcalK'_1 &= (K_1\circ F)^+,\\
\mcalK'_2 &=-\mcalK'_1 \circ K_2 \circ \left( (F \circ \mcalK_1) \otimes (F \circ \mcalK_1)\right),\\
&\dots\\
\mcalK_n'&=-\sum_{j=1}^{n-1} \mcalK'_j \circ \left(\sum_{i_1+\dots i_j = n} K_{i_1}\otimes K_{i_2} \otimes \dots \otimes K_{i_j} \right)\circ\left( (F\circ \mcalK_1) \otimes \dots \otimes (F\circ \mcalK_1) \right).
\end{split}
\end{equation}
\end{proof}

We observe that bounds on the radius of convergence, truncation error, and stability of the modified inverse Born series can be easily obtained using arguments similar to those made in Section \ref{sec:IB_o}.
Theorem \ref{thm_mod_inv_brn} can easily be applied to incorporate measurements from multiple values of $\alpha_0,$ provided the vertex potential $\eta$ is independent of the value of $\alpha_0.$ In particular, let $\Gamma= (E,V)$ be a graph and suppose we have measurements for $\alpha_0 = (\alpha_i)_1^m.$ Let $\Gamma' =\{\Gamma_1,\dots,\Gamma_m\}$ be the graph with vertices $V'= \{V_1,\dots,V_m\}$ and edges $E'=\{E_1,\dots,E_m\},$ consisting of $m$ copies of $\Gamma.$ Here the subscript denotes the copy of $E,$ $V,$ or $\Gamma$ to which we are referring. Let $\pi:V'\rightarrow V$ denote the projection map taking a vertex in $V_i$ or $\delta V_i$ to the corresponding vertex in $V$ or $\delta V,$ respectively. Finally, for a given vertex potential, $\eta,$ on $\Gamma$ let $\eta'$ denote the corresponding potential on $\Gamma'.$ Thus, for each vertex $v \in V',$
\begin{equation}
\eta'(v) = \eta(\pi(v)).
\end{equation}

Next we construct the following modified time-independent diffusion equation
\begin{equation}\label{eq:sys_mod_2}
\begin{split}
L_i u_i + \alpha_i [I+ D_{{\eta}'}] u_i -(A_i)_{V,\delta V}^Tv_i = \bf{0},\\
-(A_i)_{V,\delta V}u_i+Dv_i = \bf{g}_i,
\end{split}
\end{equation}
where $u_i$ and $v_i,$ $i=1,\dots,m,$ are supported on $V_i$ and $\delta V_i,$ respectively, and $L_i$ is the Laplacian corresponding to the $i$th subgraph. As before $D_{\eta'}$ denotes the diagonal matrix with entries given by $\eta'.$

Note that $\Gamma'$ consists of $m$ disconnected components, and hence the solution in one component is independent of the solution in another. If $W,U \subset V_i \times \delta V_i$ let $G_i^{W;U}$ denote the submatrix of $G_i$ consisting of the rows indexed by $W$ and the columns indexed by $U.$ It follows that the background Green's function for (\ref{eq:sys_mod_2}) is given by
\small
\begin{equation}
G_0=
\left(\begin{array}{cccc;{2pt/2pt}cccc}
G_1^{V_1;V_1}&&&&G_1^{V_1;\delta V_1}\\
&G_2^{V_2;V_2}&&&&G_2^{V_2;\delta V_2}\\
&&\ddots&&&&\ddots\\
&&&G_m^{V_m;V_m}&&&&G_m^{V_m;\delta V_m}\\
\hdashline[2pt/2pt]
G_1^{\delta V_1;V_1}&&&&G_1^{\delta V_1;\delta V_1}\\
&G_2^{\delta V_2;V_2}&&&&G_2^{\delta V_2;\delta V_2}\\
&&\ddots&&&&\ddots\\
&&&G_m^{\delta V_m;V_m}&&&&G_m^{\delta V_m;\delta V_m}\\
\end{array}
\right)
\end{equation}
\normalsize
Thus, if ${\bf u} = (u_1 ,\dots, u_m,v_1,\dots,v_m)^T$ solves (\ref{eq:sys_mod_2}) when $\tilde{\eta} \equiv 0,$ and 
$${\bf g} = (0,\dots,0,g_1,\dots,g_m)^T,$$
 then
\begin{equation}
{\bf u} = G_0 {\bf g}.
\end{equation}
Using this we can define the operators $K_1, \dots, K_n$ for (\ref{eq:sys_mod_2}), where we replace $G_0$ by 
\footnotesize
\begin{equation}
\begin{split}
G_0' = 
&\left(\begin{array}{cccc;{2pt/2pt}cccc}
\frac{\alpha_1}{\alpha_0}G_1^{V_1;V_1}&&&&\frac{\alpha_1}{\alpha_0}G_1^{V_1;\delta V_1}\\
&\frac{\alpha_2}{\alpha_0}G_2^{V_2;V_2}&&&&\frac{\alpha_2}{\alpha_0}G_2^{V_2;\delta V_2}\\
&&\ddots&&&&\ddots\\
&&&\frac{\alpha_m}{\alpha_0} G_m^{V_m;V_m}&&&&\frac{\alpha_m}{\alpha_0}G_m^{V_m;\delta V_m}\\
\hdashline[2pt/2pt]
\frac{\alpha_1}{\alpha_0}G_1^{\delta V_1;V_1}&&&&\frac{\alpha_1}{\alpha_0}G_1^{\delta V_1;\delta V_1}\\
&\frac{\alpha_2}{\alpha_0}G_2^{\delta V_2;V_2}&&&&\frac{\alpha_2}{\alpha_0}G_2^{\delta V_2;\delta V_2}\\
&&\ddots&&&&\ddots\\
&&&\frac{\alpha_m}{\alpha_0}G_m^{\delta V_m;V_m}&&&&\frac{\alpha_m}{\alpha_0}G_m^{\delta V_m;\delta V_m}\\
\end{array}\right),
\end{split}
\end{equation}
\normalsize
to account for the different $\alpha$ value in each component.

\begin{figure}[t]
  \centering
    \includegraphics[width=.5\textwidth]{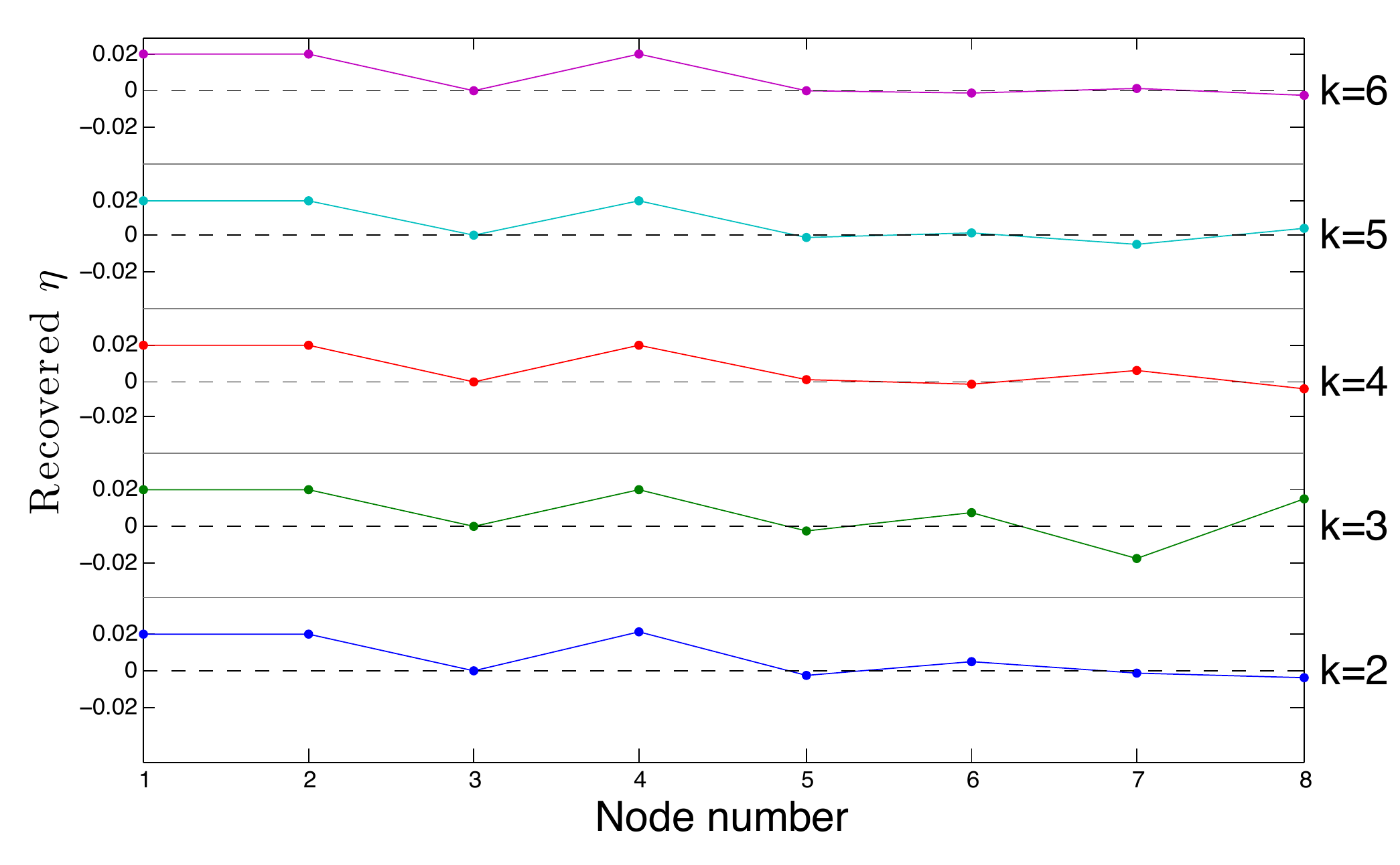}
      \caption{Multi-frequency inverse Born numerical experiments for $\alpha_i = 0.1+0.15\cdot i^\frac{1}{2},$ and $t = 0.1,$ where $\eta$ is equal to 0.02 on vertices 1, 2 and 4, and is identically zero on the remaining vertices. Here $k$ is the number of terms of the inverse Born series that were kept.} \label{fig:path_rec}
\end{figure}

We now enforce the condition that $\eta$ is identical on each copy of $\Gamma,$ and hence is independent of $\alpha.$ The map $F: \ell^p(V_1) \rightarrow \ell^p(V_1\times\dots\times V_m)$ in (\ref{eq:mod_F_def}) is defined by
\begin{equation}
F\{\eta\}(v) = \eta(\pi(v)).  
\end{equation}
Using this we form the modified inverse Born series operators in (\ref{eq:mod_inv_op}) and thus construct the modified inverse Born series. Provided that $(\mcalK_1 \circ F)$ is invertible and the measured data $\phi$ is sufficiently small, by Theorem \ref{thm_complex} the inverse Born series converges to the true (unique) value of $\eta.$ Since $\eta$ is the $\alpha$-independent absorption of the vertices in $\Gamma,$ we have constructed a reconstruction algorithm using data from multiple $\alpha_0$. To illustrate this algorithm we consider a path of length $10,$ noting that it cannot be imaged using the standard inverse Born series, that is with one value of $\alpha_0$. More generally, any graph containing a path of length greater than six in its interior, connected to the remainder of the graph only at its endpoints, the corresponding $K_1$ is not invertible. In fact, it can be shown that for such graphs that the absorption $\eta$ cannot be uniquely determined from the data $\phi.$ 

For our example, we choose one boundary vertex to act both as source and receiver and take $\alpha_i = 0.1+0.15 \cdot i^\frac{1}{2},$ $i=1,\dots, 25.$ Here $\eta$ is chosen to be a function supported on three randomly-chosen interior vertices, with a height of $0.02.$ The sums of the first few terms of the inverse Born series are shown in Figure \ref{fig:path_rec}.

\section{Acknowledgements}
This work was supported in part by the National Science Foundation grants DMS-1115574, DMS-1108969 and DMS-1619907 to JCS, and National Science Foundation grants CCF-1161233 and CIF-0910765 to ACG.

\newpage

\bibliography{bibfile}
\bibliographystyle{plain}

\end{document}